\documentclass[12pt]{amsart}
\usepackage{amssymb,amsmath,amsfonts,amsthm,bbm,mathrsfs,times,graphicx,color,comment,mathabx,subcaption}
\usepackage{tikz}
\usetikzlibrary{arrows,shapes}
\usetikzlibrary{fadings}
\usetikzlibrary{intersections}
\usetikzlibrary{decorations.pathreplacing}
\usepackage[makeroom]{cancel}

\newcommand{\BR}{\mathbb{R}}

\DeclareMathOperator{\re}{\mathbb{R}e}
\DeclareMathOperator{\im}{\mathbb{I}m}
\newcommand{\norm}[1]{\left\|#1\right\|}
\renewcommand{\i}{\mathrm i}
\newcommand{\ol}{\overline}
\newcommand{\ds}{\displaystyle}

\newcommand{\del}{\partial}

\newcommand{\sph}{{{\mathbf S}^ 1}}
\newcommand{\tta}{\theta}
\newcommand{\fii}{{\varphi}}
\newcommand{\OM}{\Omega}
\newcommand{\Gam}{\varGamma}
\newcommand{\INF}{{\infty}}
\newcommand{\bu}{{\bf u}}
\newcommand{\bv}{{\bf v}}

\newcommand{\B}{\mathcal{B}}  %% Bukhgeim-Cauchy Operator
\newcommand{\HT}{\mathcal{H}} %% Bukhgeim-HIlbert Operator

\newtheorem{theorem}{Theorem}[section]
\newtheorem{prop}{Proposition}[section]
\newtheorem{lemma}{Lemma}[section]

\title[Partial inversion of the 2D attenuated $X$-ray transform]{Partial inversion of the 2D attenuated  $X$-ray transform with data on an arc}
\begin{document}
\date{\today}
\author{Hiroshi Fujiwara}
\address{Graduate School of Informatics,  Kyoto University, Yoshida Honmachi, Sakyo-ku, Kyoto 606-8501, Japan }
\email{fujiwara@acs.i.kyoto-u.ac.jp}
\author{Kamran Sadiq}
\address{Johann Radon Institute of Computational and Applied Mathematics (RICAM), Altenbergerstrasse 69, 4040 Linz, Austria}
\email{kamran.sadiq@ricam.oeaw.ac.at}
%    Information for last author
\author{Alexandru Tamasan}
\address{Department of Mathematics, University of Central Florida, Orlando, 32816 Florida, USA}
\email{tamasan@math.ucf.edu}
%    General info
\maketitle

%%%%%%%%%%%%%%%  ABSTRACT %%%%%%%%%%%%%%%%%%%%%%%%%%
\begin{abstract}
In two dimensions, we consider the problem of inversion of the attenuated $X$-ray transform of a compactly supported function from data restricted to lines leaning on a given arc. We provide a method to reconstruct the function on the convex hull of this  arc. The attenuation is assumed known. % in this subdomain. 
The method of proof uses the Hilbert transform associated with $A$-analytic functions in the sense of Bukhgeim. 
\end{abstract}

%%%%%%%%%%%%%%%% INTRODUCTION %%%%%%%%%%%%%%%%%%%%%%%
\section{Introduction}

This work concerns  a problem of the inversion of the attenuated $X$-ray transform from partial measurements in two dimensional domains. Discriminated by the type of their partial data, 
such problems have already been considered both in the non-attenuated and the attenuated case. For example the $180^0$ angular data in \cite{rullgard,boman06},  or the partial data case in \cite{bal}; or, in the  constant attenuation case in the works \cite{nooWagner01,panKM02,rullgard04,kuchmentShneiberg,nooDPC07}. 

In here we are concerned with yet a different question, where the (fan beam) data is collected on a convex arc. More precisely, let $\Omega \subset \mathbb{R}^2$ be a strictly convex domain and $\Lambda$ be an arc of its boundary $\Gam$. The chord $L$ joining the endpoints of the arc $\Lambda$ partitions the domain in two subdomains $\Omega^\pm$, where $\OM^+$ denotes the domain enclosed by $\Lambda\cup L$,  see Figure \ref{fig:1} below. For a function $f$ compactly supported in $\OM$, we provide a method to reconstruct $f\lvert_{\OM^+}$ from its attenuated $X$-ray transform over lines leaning on $\Lambda$. In the attenuated case, the attenuation $a$ is assumed known in $\ol\OM$. This is a particular type of the region-of-interest (ROI) inversion problem in tomography \cite{clackdoyleDefrise10}; see \cite{boman16} for some limitations on the reconstruction methods of such problems.

Even in the non-attenuated case, if $f$ happens to also be supported in $\OM^-$, then its $X$-ray data is incomplete and the classical Radon inversion \cite{helgasonBook,radon1917} does not apply: while the measurements are affected by the possible nonzero values of $f|_{\OM^-}$, an entire cone of directions through points in $\OM^-$ are missing in the data, yielding the reconstruction question nontrivial, even for the restriction  $f|_{\Omega^+}$.

The source reconstruction problem here is not covered by the inversion from partial data results mentioned above. However,  the unique determination of $f\lvert_{\Omega^+}$ follows from the support theorem in \cite{bomanQuinto}, provided the attenuation is analytic. One could also adapt the Carleman weight formula in \cite{arbuzovBukhgeim97} by finding an explicit Carleman quenching operator as in \cite{fktRIMS21}. Both such arguments are based on analytic continuation and have not yet yielded a method of reconstruction. The novelty in this work is a reconstruction method, based on the separation of the contribution of $f\lvert_{\Omega^+}$ from that of $f\lvert_{\OM^-}$, and, in the attenuated case, does not assume analyticity of the attenuation coefficient.

% If $f$ is known to be supported in $\OM$, then via a conformal mapping of $\OM$ onto the upper half plane the original arguments based on a Carleman weight formula  in \cite{arbuzovBukhgeim97,arbuzovBukhgeim01} can be modified to prove unique determination of $f$;  see the section Preliminaries  for an explicit Carleman weight-operator tailored for $\OM^+$. 

%%%%%%%%%%%% FIGURE OF THE DOMAIN %%%%%%%%%%%%%%%%%%%
\begin{figure}[ht]
\centering
\begin{tikzpicture}[scale=1.14, cap=round,>=latex]
  % Drawing Omega and the source
      \draw[name path=ellipse, thick,black] (0cm,1cm) ellipse (1.6cm and 2cm);
      %\draw[name path=ellipse, thick,blue] (0.30cm,1.25cm) ellipse (1.0cm and 1.1cm);
      \draw[name path=ellipse, thick,blue] (0.35cm,1.35cm) ellipse (1.0cm and 1.1cm);
  % Source f
      \coordinate[label=above:$\mathbin{\color{blue}{f}}$] (source) at (90:1.8cm);
  % Omega plus and minus
      \coordinate[label=above:$\mathbin{\color{red}{\Omega^{+}}}$] (OMp) at(100:2.4cm);
      \coordinate[label=above:$\Omega^{-}$] (OMm) at(280:0.75cm);
  % Gamma
      \coordinate[label=above:$\mathbin{\color{black}{\Gam = \del \OM}}$] (Gam) at(15:2.3cm);
  % Lambda
      \coordinate[label=above:$\mathbin{\color{black}{\Lambda}}$] (Lam) at(85:3.0cm); 
  % The point zeta
      \filldraw[black] (70:2.79cm) circle(1.2pt);
      \coordinate[label=left:$\zeta$] (zeta) at(68.5:2.70cm);
      \coordinate (zeta) at (68.5:2.78cm);
  % Normal at zeta
      %\draw[->] (70:2.78cm) -- (67:3.5cm);
      \draw[->] (70:2.78cm) -- (66.5:3.5cm);
      \draw [color=black](65:3.6cm) node[rotate=-35, scale=0.9] {$\nu(\zeta)$};
  % The angle theta and the line 
      \coordinate[label=above:$\tta$] (tta) at(74:3.6cm);
      \draw[->] (110:-2.0cm)--(74:3.65cm);
  % The endpoints of Gamma and L
      \filldraw[red] (140:2.08cm) circle(1.2pt);
      \coordinate (endpoint1) at(140:2.10cm);
      \filldraw[red] (41:2.1cm) circle(1.2pt);
      \coordinate (endpoint2) at (41:2.12cm);
      \coordinate[label=below:$\mathbin{\color{red}{L}}$] (Lsegment) at (80:1.4cm);
  % Line connecting the endpoints of the Gamma
      \draw[color=red] (endpoint1) -- (endpoint2);
      \tikzset{position label/.style={below = 3pt, text height = 1.5ex, text depth = 1ex},
               brace/.style={decoration={brace,mirror},decorate}}
  % Drawing the length of the chord from z_{0} to Omega
      \tikzset{position label/.style={above = 2.5pt,text height = 1.1ex,text depth = 1ex},
               brace/.style={decoration={brace,mirror}, decorate}}
\end{tikzpicture}
\caption{Geometric setup: $\del{\OM^+} =\Lambda\cup {\color{red}L}$} \label{fig:1}
\end{figure}

%%%%%%%%%%%%%%%%%%%%%%%%%%%%%%%%%%%%%%%%%%%%%%%%%%
For a real valued function $a\in L^1(\BR^2)$, the attenuated $X$-ray transform of $f$ is given by, 
\begin{align}\label{Xaf}
X_{a}f(z,\tta) := \int_{-\infty}^{\infty} f(z+s\tta) e^{-Da(z+s\tta, \tta)}ds, \quad (z,\tta) \in \OM\times \sph,
\end{align} where 
\begin{align}\label{divbeam}
 Da(z,\tta) := \int_{0}^{\INF}a(z+t\tta)dt 
\end{align} is the divergence beam transform of $a$.
For the non attenuated case $a\equiv 0$ we use the notation $Xf$.

In here, the inversion problem is approached through the known equivalence between the attenuated $X$-ray transform and the boundary value problems for the transport equation: Let $\Gam_\pm:=\{(\zeta,\tta)\in \Gam \times\sph:\, \pm\nu(\zeta)\cdot\tta>0 \}$ denote the outgoing $(+)$, respectively incoming $(-)$ submanifolds of the unit tangent bundle of $\Gam$,
with $\nu(\zeta)$ being the outer normal at $\zeta \in \Gam$ and $\tta$ is a direction in the unit sphere $\sph$. If $u(z,\tta)$ is the unique solution to 
\begin{subequations} \label{AttenRadonTEq}
\begin{align} \label{AttenRadonTEq_in_u}
  \tta\cdot\nabla u(z,\tta) +a(z) u(z,\theta) &= f(z)  \quad (z,\tta)\in \OM\times\sph, \\ \label{UGamMinus}
   u|_{\Gam_-} &= 0, 
\end{align}
\end{subequations}
then its trace on $\Gam_+$ satisfies
\begin{align}\label{u_Gam+}
u \lvert_{\Gam_{+}} (\zeta,\tta) = X_a f (\zeta,\tta), \quad (\zeta,\tta)\in \Gam_{+}.
\end{align}In our problem, the data $X_a f$ is only available on 
\begin{align*}
\Lambda_+:=\{(\zeta,\tta)\in \Lambda\times\sph:\, \nu(\zeta)\cdot\tta>0 \}.
\end{align*}

In Section \ref{partialInv} we prove the following.
\begin{theorem}\label{reconstruct_f_OM+}
Let $\OM\subset\mathbb{R}^2$ be a strictly convex, domain with $C^2$ boundary, and $\Lambda \subset\partial\OM$ be an arc of its boundary. Let $\Omega^+\subset\OM$ be the subset enclosed by $\Lambda$ and by the chord $L$ joining the endpoints of $\Lambda$. For $\mu>1/2$, assume that  $a \in C^{1,\mu}(\ol\OM)$ is a known real valued function in $\OM$.
If $f \in C_0^{1,\mu}(\OM)$ is an unknown real valued function, then its restriction in ${\OM^+}$ can be reconstructed from the partial data $X_{a}f|_{\Lambda_+}$. \end{theorem}
The method of reconstruction stated in the result above is yet another application of the theory of $A$-analytic functions originally developed by Bukhgeim \cite{bukhgeim_book} to address the tomography problem from complete data, see \cite{ABK} for the attenuated case. For different approaches to the inversion of the attenuated $X$-ray transform from complete data we refer to \cite{novikov01,novikov02}, and further developments in \cite{naterrer01, bomanStromberg, bal, kazantsevBukhgeimJr07, monard17}. 
 
In the preliminary section we recall some ingredients used in the partial reconstruction method. Section \ref{partialInv} contains the new idea of transporting the $X$-ray data from the boundary arc $\Lambda$ to the interior segment $L$. This is a direct consequence of the range characterization condition  in terms of the Bukhgeim-Hilbert transform associated with $A$-analytic maps, see \cite{sadiqtamasan01,sadiqtamasan02,sadiqscherzertamasan}. It is also in Section \ref{partialInv} that we reconstruct $f\lvert_{\Omega^+}$, and prove Theorem  \ref{reconstruct_f_OM+}. The proof is constructive and yields a (partial) inversion that is implemented in the numerical experiments of Section \ref{numer_recon}.

%%%%%%%%%%%%%%%%%%%%%%%%%%%%%%%%%%%%%%%%%%%%%%%%%%%%%%%%%
\section{Preliminaries}
In this section we recall some known properties of the $A$-analytic functions and of the finite Hilbert transform, on which 
our reconstruction method is based. We denote by $l_\infty$, and by $l_1$, the space of bounded, respectively summable, sequences. For $z=x+\i y$, let 
$\ol{\del} = \left( \del_{x}+ \i \del_{y} \right) /2 $, and $\del = \left( \del_{x}- \i \del_{y} \right) /2$
be the Cauchy-Riemann operators. 

A sequence valued map
$$\OM \ni z\mapsto  \bu(z): = \langle u_{0}(z), u_{-1}(z),u_{-2}(z),... \rangle$$ in $C(\ol\OM;l_\INF)\cap C^1(\OM;l_\INF)$
is called {\em $\mathcal{L}$-analytic}, if
\begin{equation}\label{Aanalytic}
\ol{\del} \bu (z) + \mathcal{L} \del \bu (z) = 0,\quad z\in\OM,
\end{equation} where  $\mathcal{L}$ is the left shift operator,
$\displaystyle\mathcal{L}  \langle u_{0}, u_{-1}, u_{-2}, \cdots  \rangle =  \langle u_{-1}, u_{-2},  \cdots \rangle.$
Note that we use the sequences of non-positive indexes to conform with the original notation in Bukhgeim's work \cite{bukhgeim_book}.

For completeness we recall  the spaces referenced in the results below. For $0<\mu<1$, and  $\Gam$ (some part of) the boundary $\del \OM$ we consider the spaces of sequence valued maps
\begin{equation*}
 \begin{aligned} 
 l^{1,1}_{\INF}(\Gam) &:= \left \{ \bu\; : \sup_{\xi \in \Gam}\sum_{j=1}^{\INF}  j \lvert u_{-j}(\xi) \rvert < \INF \right \},\\
 C^{\mu}(\Gam; l_1) &:= \left \{ \bu:
\sup_{\xi\in \Gam} \lVert \bu(\xi)\rVert_{\ds l_{1}} + \underset{{\substack{
            \xi,\eta \in \Gam \\
            \xi\neq \eta } }}{\sup}
 \frac{\lVert \bu(\xi) - \bu(\eta)\rVert_{\ds l_{1}}}{|\xi - \eta|^{ \mu}} < \INF \right \}.
 \end{aligned}
 \end{equation*}
We similarly consider  $C^{\mu}(\overline\OM ; l_1)$, and $C^{\mu}(\ol\OM ; l_{\INF})$.

Analogous to the analytic maps,  the $\mathcal{L}$-analytic maps  are determined by their boundary values via a Cauchy-like integral formula \cite{bukhgeim_book}. Following \cite{finch}, the Bukhgeim-Cauchy operator $\B$ acting on $\bu=\langle u_{0}, u_{-1}, u_{-2},...\rangle \in l^{1,1}_{\INF}(\Gam)\cap C^\mu(\Gam;l_1)$
is defined component-wise for $n\leq 0$ by
\begin{align} \label{BukhgeimCauchyFormula}
(\B \bu)_{n}(z) &:= \frac{1}{2\pi \i} \int_{\Gam}
\frac{ u_{n}(\zeta)}{\zeta-z}d\zeta  + \frac{1}{2\pi \i}\int_{\Gam} \left \{ \frac{d\zeta}{\zeta-z}-\frac{d \ol{\zeta}}{\ol{\zeta-z}} \right \} \sum_{j=1}^{\infty}  
 u_{n-j}(\zeta)
\left( \frac{\ol{\zeta-z}}{\zeta-z} \right) ^{j},\; z\in\OM;
\end{align} The explicit form above follows \cite{finch}. 

Also similar to the analytic maps, the traces on the boundary of $\mathcal{L}$-analytic maps satisfy some constraints, which can be expressed in terms of a corresponding Hilbert transform introduced in  \cite{sadiqtamasan01}. More precisely, the Bukhgeim-Hilbert transform $\HT$ acting on 
$\bu\in l^{1,1}_{\INF}(\Gam)\cap C^\mu(\Gam;l_1)$ is defined 
component-wise for $n \leq 0$ by
 \begin{align}\label{hilbertT}
(\HT\bu)_{n}(\xi)&=\frac{1}{\pi} \int_{\Gam } \frac{u_{n}(\zeta)}{\zeta - \xi} d\zeta 
+\frac{1}{\pi} \int_{\Gam } \left \{ \frac{d\zeta}{\zeta-\xi}-\frac{d \ol{\zeta}}{\ol{\zeta-\xi}} \right \} \sum_{j=1}^{\infty}  
u_{n-j}(\zeta)
\left( \frac{\ol{\zeta-\xi}}{\zeta-\xi} \right) ^{j},\; \xi\in\Gam.
\end{align}
The first integral is in the sense of principal value, e.g., \cite{muskhellishvili} . 

%%%%%%%%%%%%%%%%%%%%%%%%%%%%%%%%%%%%%
\begin{theorem}\label{BukhgeimCauchyThm}
Let $0<\mu<1$. Let $\bu = \langle u_{0}, u_{-1}, u_{-2},...\rangle$ be a sequence valued map defined on the boundary $\Gam$ and $\B$ be the Bukhgeim-Cauchy operator acting on $\bu$ as in \eqref{BukhgeimCauchyFormula}.   If $\bu \in l^{1,1}_{\INF}(\Gam)\cap C^\mu(\Gam;l_1)$, then $\B \bu\in C^{1,\mu}(\OM;l_\infty)\cap C(\ol \OM;l_\infty)$ is $\mathcal{L}$-analytic in $\OM$. If  $\bu \in C^{1,\mu}(\OM;l_\infty)\cap C(\ol \OM;l_\infty)$ is $\mathcal{L}$-analytic in $\ol\OM$, then $\bu(z)=\B\b(z)$, for $z \in \ol \OM$.

Moreover, for $\bu$ to be the boundary value of an $\mathcal{L}$-analytic function it is necessary and sufficient that
\begin{align} \label{NecSufEq}
   (I+\i\HT) \bu = {\bf {0}}.
\end{align} 
\end{theorem}
For the proof of first statement of the Theorem we refer to \cite[Theorem 2.2]{sadiqtamasan02}. For the proof of the last statement of the Theorem we refer to \cite[Theorem 3.2]{sadiqtamasan01}. 
%%%%%%%%%%%%%%%%%%%%%%%%%%%%%%%%

The method of reconstruction proposed here considers the operator $ [I - \i H_t]$, where $H_t$ is the finite  Hilbert transform 
%We will also make use of the finite Hilbert transform
\begin{align}\label{finite_Hilbert}
H_t g(x)=\frac{1}{\pi}\int_{-l}^{l}\frac{g(s)}{x-s}ds, \quad x \in (-l,l),
\end{align} with the integral understood in the sense of principal value. 
It is well-known, e.g \cite{tricomi57}, that $\i H_t$ is a bounded operator in $L^2(-l,l)$, with spectrum $[-1,1]$, see \cite{koppelmanPincus58,okadaElliott91}. However, $1$ is not in the point spectrum, see e.g., \cite{widom60} for a proof based on Riemann-Hilbert problem. The arguments below use the unitary property in $L^2(\BR)$ of the (infinite) Hilbert transform 
\begin{align}\label{classical_Hilbert}
H f(x)=\frac{1}{\pi}\int_{-\INF}^{\INF}\frac{f(s)}{x-s}ds.
\end{align}
%%%%%%%%%%%%%%%%%%%
\begin{prop}\label{prop_finiteHilbert}In $L^2(-l,l)$,
 \begin{align}\label{kernel_Pminus}
 Ker [I-\i H_t] = \{ 0\}.
\end{align}
\end{prop}
%%%%%%%%%%%%%%%%%%%%
\begin{proof}
Let $f\in L^2(-l,l)$ be extended by zero to the entire real line.
If $f \in Ker [I-\i H_t]$, then 
$|f| = |H_tf|=|Hf|$ a.e. in $(-l,l)$, and 
\begin{align*}
 \norm{f}_{L^2(-l,l)}^2 &= \norm{f}_{L^2(\BR)}^2 = \norm{Hf}_{L^2(\BR)}^2 
%  \\ &
 = \int_{-l}^{l} |H_tf|^2 dx + \int_{|x| >l} |Hf(x)|^2 dx \\
 &= \int_{-l}^{l} |f|^2 dx + \int_{|x| >l} |Hf(x)|^2 dx.
\end{align*}The latter identity shows that $|Hf(x)|=0$, for $|x| >l$. Thus, for $|x|>l$,
%$H f(x)$ is an analytic function in $\ds \frac{1}{x}$, 
\begin{align*}
 0 &=\int_{-l}^{l}\frac{f(s)}{x-s}ds = \frac{1}{x} \int_{-l}^{l} f(s) \sum_{j=0}^{\INF} \left(\frac{s}{x} \right)^jds =   \sum_{j=0}^{\INF} \frac{1}{x^{j+1}} \int_{-l}^{l} f(s) s^jds,
\end{align*}yielding $f$ orthogonal on any polynomial. Since polynomials are dense in $L^2(-l,l)$, $f=0$.
\end{proof}
%%%%%%%%%%%%%%%%%%%%%%%%%%%%%%%%%%%%%%%%%%%%%%%%

%%%%%%%%%%%%%%%%%%%%%%%%%%%%%%%%%%%%%%%%%%%%%%%%%%%%%%%%% 

\section{Partial source reconstruction}\label{partialInv}
%%%%%%%%%%%%%%%%%%%%%%%%%%%%%%%%%%%%%%%%%%%%%%%%%%%%%%
%%%%%%%%%%%%%%%%%%%%%%%%%%%%%%%%%%%%%%%%%%%%%%%%%%%%%%%%%
In this section, we present the method of reconstruction of $f|_{\OM^+}$. We start with  the non-attenuated case, $a=0$. Upon a rotation and translation of the domain $\OM$, we assume without loss of generality that the arc $\Lambda$ lies in the upper half plane with the endpoints on the real axis lying symmetrically about the origin. In particular, $\OM \cap \{  \im{z}=0\}=L= (-l,l)$, for some $l>0$, and $\partial\OM^+=\Lambda\cup L$.

The main idea in our reconstruction method is the recovery of the trace on $L$ of an $\mathcal{L}$-analytic function from its trace on $\Lambda$. In the following result, the arc $\Lambda$ and the chord $L$ are considered without their endpoints.

%%%%%%%%%%%%%%%%%%%%%%%%%%%%%%%%%%%%%%%%%%%%%%%%%%%%%%%%%%%%%%%%%%%%%%
To simplify the statement, for each $n\leq 0$, let us introduce the functions $F_n(z)$ defined on $\Lambda\cup L$ except at $\pm l$, by
\begin{align}\label{Fanalytic}
F_n(z):= \frac{1}{\i \pi }\int_\Lambda\frac{u_n(\zeta)}{\zeta-z}d\zeta 
+\frac{1}{\i \pi }\int_\Lambda\left\{\frac{d\zeta}{\zeta-z}-\frac{d\ol\zeta}{\ol\zeta-\ol z}\right\}\sum_{j=1}^\infty u_{n-j}(\zeta)\left(\frac{\ol\zeta-\ol z}{\zeta-z}\right)^j,\;z\in\Lambda \cup L,
\end{align} where the first integral is in the sense of principal value. Note that $F_n$ is known since $u_n\lvert_\Lambda$ are known for all $n\leq 0$.  
\begin{theorem}\label{newAnalyticTh}
Let $\OM\subset\mathbb{R}^2$ be a strictly convex bounded domain with $C^2$-smooth boundary $\Gam$, $\Lambda \subset\Gam$ be an arc of its boundary, and $L$ be the chord  joining the endpoints of $\Lambda$. Let $\Omega^+$ be the subset of $\OM$ enclosed by $\Lambda\cup L$ and $0<\mu<1$.  If  $\bu  \in l^{1,1}_{\INF}(\del \OM^+)\cap C^\mu(\del \OM^+;l_1)$ is the boundary value of an $\mathcal{L}$-analytic function in $\OM^+$, 
then its trace $\bu \lvert_{L}$  on $L$ is recovered pointwise from the trace $\bu \lvert_{\Lambda}$ on $\Lambda$ as follows:
for each $n\leq 0$, $u_n \big \lvert_{L}$ is the unique solution in $L^2(-l,l)$ of  
\begin{align}\label{Pminus_un}
[I - \i H_t](u_n)(x)=& F_{n}(x), \quad x \in L,
\end{align}where $H_t$ is the finite  Hilbert transform in \eqref{finite_Hilbert}, and 
the right hand side is determined by the data from \eqref{Fanalytic}.
\end{theorem}

%%%%%%%%%%%%%%%%%%%%%%%%%%%%%%%%%%%%%%%%%%%%%%%%%%%%%%%%%%%%
\begin{proof}

Since $\bu  \in l^{1,1}_{\INF}(\del \OM^+)\cap C^\mu(\del \OM^+;l_1)$ is the boundary value of an $\mathcal{L}$-analytic function in $\OM^+$, then the necessity part of Theorem \ref{BukhgeimCauchyThm} yields 
\begin{align}\label{Cond1}
[I+\i\HT] \bu=0,
\end{align}where $\HT$ is the Bukhgeim-Hilbert transform in \eqref{hilbertT}.

We consider \eqref{Cond1} on $L$, where for each $x\in (-l,l)$ and $n\leq 0$, the $n$-th component yields
\begin{align}\nonumber
u_{n}(x)  - \frac{\i}{\pi} \int_{-l}^{l}  \frac{u_{n}(s)}{x-s}  ds = -\frac{\i}{\pi} \int_{\Lambda} \frac{u_{n}(\zeta)}{\zeta-x} d\zeta 
 &-\frac{\i}{\pi} \int_{\Lambda } \left \{ \frac{d\zeta}{\zeta-x}-\frac{d \ol{\zeta}}{\ol{\zeta}-x} \right \}
  \sum_{j=1}^{\infty}  u_{n-j}(\zeta)
\left( \frac{\ol{\zeta}-x}{\zeta-x} \right) ^{j} \\
&-
%\cancel{
\frac{\i}{\pi} 
\int_{-l}^{l} \left \{ \frac{d\zeta}{\zeta-x}-\frac{d \ol{\zeta}}{\ol{\zeta}-x} \right \}
  \sum_{j=1}^{\infty}  u_{n-j}(\zeta)
\left( \frac{\ol{\zeta}-x}{\zeta-x} \right) ^{j}
%}++A
.\label{knowntobezero}
\end{align}

Since the last integral in \eqref{knowntobezero} ranges over the reals, it vanishes. The remaining integrals give $F_n(x)$, and the expression \eqref{Pminus_un} follows.
\end{proof}

The equation \eqref{Pminus_un} may not have any solution for an arbitrary right hand side in $L^2(-l,l)$. However, in the inverse problem, the function $F_n$ already belongs to the range of $I -\i H_t$, so that the solution exists. Moreover, since the range is open, \eqref{Pminus_un} is uniquely solvable in a sufficiently small  $L^2$-neighborhood of $F_n$.
 
%%%%%%%%%%%%%%%%%%%%%%%%%%%%%%%%%%%%%
We remark that $\bu \lvert_{L}$ satisfies further constraints. If we consider \eqref{Cond1} on $\Lambda$, we obtain for each $z\in \Lambda$ and $n\leq 0$:
% the $n$-th component yields
\begin{align*}\nonumber
u_{n}(z)  - \frac{\i}{\pi} \int_{-l}^{l}  \frac{u_{n}(s)}{z-s}  ds = -\frac{\i}{\pi} \int_{\Lambda} \frac{u_{n}(\zeta)}{\zeta-z} d\zeta 
 &-\frac{\i}{\pi} \int_{\Lambda } \left \{ \frac{d\zeta}{\zeta-z}-\frac{d \ol{\zeta}}{\ol{\zeta-z}} \right \}
  \sum_{j=1}^{\infty}  u_{n-j}(\zeta)
\left( \frac{\ol{\zeta-z}}{\zeta-z} \right) ^{j} \\
&-\frac{\i}{\pi} \int_{-l}^{l} \left \{ \frac{1}{s-z}-\frac{1}{s-\ol{z}} \right \}
  \sum_{j=1}^{\infty} u_{n-j}(s) \left( \frac{s-\ol{z}}{s-z} \right) ^{j} ds, %\label{uneqL}
\end{align*}
or, by \eqref{Fanalytic},
%The above system \eqref{uneqL} yields that $\bu|_L$ satisfies additional constraints for each $z\in \Lambda$ and $n\leq 0$,
\begin{align}\label{integralEq2}
 &\int_{-l}^{l} \frac{u_{n}(x)}{x-z}  dx + \int_{-l}^{l}\frac{(z-\ol{z})}{|x-z|^2} 
  \sum_{j=1}^{\infty} u_{n-j}(x) \left( \frac{x-\ol{z}}{x-z} \right) ^{j} dx =  \pi \i \left[ u_{n}(z) - F_{n}(z) \right],\;\;z\in\Lambda. %, \;z\in\Lambda,\;n\leq 0.
\end{align}
The additional constraints  \eqref{integralEq2} may be used to stabilize the numerical inversion. However, in the numerical experiments below the only regularization used is due to the discretization.

% From a numerical perspective it is important to note that the integrals in the right hand side of \eqref{integralEq} are not singular for $z$ on the arc $\Lambda$, away from its endpoints.
% %%%%%%%%%%%%%%%%%%%%%%%%%%%%%%%%%%%%%%%%%%%%%%%%%%%%%%%

Since $\bu$ is now known on the entire boundary $\partial{\OM^+}=\Lambda\cup L$, we employ the Bukhgeim-Cauchy integral formula \eqref{BukhgeimCauchyFormula} to recover 
$u_{-1}$, and then the source  in $\OM^+$, by
\begin{align}\label{recon}
f|_{\OM^+}(z)=2\re\{\del [\B \bu]_{-1}(z)\},\quad z\in\OM^+.
\end{align}

For the sequence valued map  $\bu$ arising in the  $X$-ray transform \eqref{AttenRadonTEq}, its regularity needed in Theorem \ref{newAnalyticTh} follows directly from the regularity properties 
of solutions to the transport equation, particularly as shown in \cite[Proposition 4.1]{sadiqtamasan01}.

%%%%%%%%%%%%%%%%%%%%%%%%%%%%% ATTENUATED CASE %%%%%%%%%%%%%%%%%%%%%%%%%%%%

Next we consider the attenuated case. 
Reconstruction in the attenuated case follows by the reduction to the non-attenuated case via the special integrating factor
$e^{-h}$ with
%\begin{align}%\label{hDefn}
$$h(z,\tta) := Da(z,\theta) -\frac{1}{2} \left( I - \i H \right) Ra(z\cdot \tta^{\perp},\tta),$$
%\end{align}
In the equation above  $\tta^\perp$ is  orthogonal  to $\tta$, $Da(z,\theta)$ is the divergence beam transform in \eqref{divbeam}, $Ra(s,\tta) = \ds \int_{-\INF}^{\INF} a\left( s \tta^{\perp} +t \tta \right)dt$ is the Radon transform of the attenuation, while the Hilbert transform $H$  of  \eqref{classical_Hilbert}
is taken in the first variable and evaluated at $s = z \cdotp \tta^{\perp}$. Since 
\begin{align*}%\label{h=IntegratingFactor}
\tta \cdot \nabla h(z,\tta) = -a(z), \quad (z, \tta) \in \OM \times \sph,
\end{align*}
 $u$ is the unique solution of \eqref{AttenRadonTEq} if and only if 
\begin{align*}%\label{uv_connection}
v(z,\tta)=e^{-h(z,\tta)}u(z,\tta),  \quad  (z,\tta)\in \OM \times \sph,       
\end{align*} is the unique solution of 
\begin{subequations} \label{transp_a>0}
  \begin{align} \label{transp_a>0_in_v}
  \tta\cdot\nabla v(z,\tta) &=f(z)e^{-h(z,\tta)} \quad (z,\tta)\in \OM\times\sph, \\ \label{VGamMinus}
   v|_{\Gam_-} &= 0. 
\end{align}
\end{subequations}
In particular the sequence of their respective negative Fourier modes $u_n=\int_{\sph}u(z,\theta) e^{-\i n\theta}d\theta$ satisfy
\begin{align}\label{AfEq}
%\ol{\del} u_{1}(z) &+ \del u_{-1}(z) +a(z)u_{0}(z) = f(z), \\  \label{AnalyticEq_u}
\ol{\del} u_{n}(z) &+ \del u_{n-2}(z) +a(z)u_{n-1}(z) = 0, \quad n \leq 0,
\end{align}while for $v_n=\int_{\sph}v(z,\theta) e^{-\i n\theta}d\theta$,
\begin{align}
%\ol{\del} v_{1}(z) &+ \del v_{-1}(z) = \alpha_0(z){f}(z),\nonumber \\
\ol{\del} v_{n}(z) &+ \del v_{n-2}(z)= 0, \quad n \leq 0.\label{AnalyticEq_v}
\end{align}

The key property of $h(z,\cdot)$  is that it extends analytically from $\sph$ inside the unit disk as noted by Natterrer in \cite{naterrerBook}; see also \cite{finch, bomanStromberg}. In particular all its negative Fourier coefficients vanish, and thus
\begin{align}\label{ehEq}
  e^{- h(z,\tta)} := \sum_{k=0}^{\INF} \alpha_{k}(z) e^{\i k\fii}, \quad e^{h(z,\tta)} := \sum_{k=0}^{\INF} \beta_{k}(z) e^{\i k\fii}, \quad (z, \tta) \in \ol\OM \times \sph.
  \end{align}

%%%%%%%%%%%%%%%%%%%%%%%%%%%%%%%%%%%%%%%%%%%%%%%
Since the attenuation $a$ is assumed known in ${\ol\OM}$,  $\alpha_n$ and $\beta_n$ are known in $\ol{\OM}$ for all $n\geq 0$. Moreover, for $a\in C^{1,\mu}(\ol \OM)$, $\mu>1/2$, the sequence valued maps
%\begin{align*}&
$z\mapsto\langle \alpha_{0}(z), \alpha_{1}(z), \alpha_{2}(z),  ... , \rangle$, and $z\mapsto \langle \beta_{0}(z), \beta_{1}(z), \beta_{2}(z),  ... , \rangle$ lie in  $C^{1,\mu}(\OM ; l_{1})\cap C(\ol\OM ; l_{1})$; 
%\end{align*}
see, e.g.,   \cite[Lemma 4.1]{sadiqscherzertamasan}.

The equivalence of the attenuated to non-attenuated $X$-ray case is based on the connection between \eqref{AfEq} and \eqref{AnalyticEq_v}, which, in turn, is intrinsic to negative Fourier modes only:
\begin{lemma}\label{intrinsic} \cite[Lemma 4.2]{sadiqscherzertamasan}
Assume $a\in C^{1,\mu}(\ol \OM), \mu>1/2$ and let  $\alpha_j, \beta_j$'s be the Fourier modes in \eqref{ehEq}.

(i) If $\bv \in C^1(\OM, l_1)$ satisfy \eqref{AnalyticEq_v}, and $\bu = \langle u_{0}, u_{-1}, u_{-2}, ... \rangle$ is defined component-wise by the convolution
  \begin{align}\label{UintermsV}
 u_{n} := \sum_{j=0}^{\INF} \beta_j v_{n-j}, \quad n \leq 0,
 \end{align}then $\bu \in C^1(\OM, l_1)$ solves \eqref{AfEq}.

 (ii) Conversely, if $\bu\in C^1(\OM, l_1)$ satisfy \eqref{AfEq}, and $\bv = \langle v_{0}, v_{-1}, v_{-2}, ... \rangle$ is defined component-wise by the convolution
  \begin{align}\label{VintermsU}
 v_{n} := \sum_{j=0}^{\INF} \alpha_j u_{n-j}, \quad n \leq 0.
 \end{align}then $\bv \in C^1(\OM, l_1)$ solves \eqref{AnalyticEq_v}.
\end{lemma}

%%%%%%%%%%%%%%%%%%%%%%%%% Proof of The Main Theorem %%%%%%%%%%%%%%%%%%%%%
{\bf Proof of the Theorem \ref{reconstruct_f_OM+}}\label{proof_f_OM+.}

Recall that 
\begin{align*}
 u(z,\tta) = \sum_{n =-\INF}^{\INF} u_{n}(z) e^{\i n\fii} \quad \text{and} \quad v(z,\tta) = \sum_{n =-\INF}^{\INF} v_{n}(z) e^{\i n\fii}
\end{align*}
are the solution in $\OM \times \sph$ to the boundary value problem \eqref{AttenRadonTEq}, respectively \eqref{transp_a>0}, with $v =e^{-h}u$. Let $\bu= \langle u_{0}, u_{-1}, u_{-2}, ... \rangle $ be the sequence valued map of non-positive Fourier modes of $u$, and let $\bv^{odd} = \langle v_{-1},v_{-3},v_{-5},...\rangle$, $\bv^{even} = \langle v_0, v_{-2},v_{-4},...\rangle$ be the subsequences of negative odd, respectively, even Fourier modes of $v$.

Since $f \in C_0^{1,\mu}(\OM)$ and $a \in C^{1,\mu}(\ol\OM)$, then $u, h, v \in C^{1,\mu}(\ol\OM \times \sph)$. The traces $u \lvert_{\del \OM^+ \times \sph} \in C^{1,\mu}(\del\OM^+;C^{1,\mu}(\sph))$ and $v \lvert_{\del \OM^+ \times \sph} \in C^{1,\mu}(\del\OM^+;C^{1,\mu}(\sph))$. 
By applying \cite[Proposition 4.1 (i)]{sadiqtamasan01} we obtain $\bu, \bv^{even}, \bv^{odd}  \in l^{1,1}_{\INF}(\del \OM^+)\cap C^\mu(\del \OM^+;l_1)$. 

By \eqref{u_Gam+}, the attenuated $X$-ray transform $X_{a}f$ on $\Lambda_+$ determines $\bu$ on $\Lambda$.
By formula \eqref{VintermsU}, $\bu \lvert_{\Lambda}$  determines the traces $\bv^{odd} \lvert_{\Lambda} \in l^{1,1}_{\INF}(\Lambda)\cap C^\mu(\Lambda;l_1)$ and $\bv^{even}\lvert_{\Lambda} \in l^{1,1}_{\INF}(\Lambda)\cap C^\mu(\Lambda;l_1)$ on $\Lambda$.

The equations \eqref{AnalyticEq_v} show that $\bv^{odd}$ and $ \bv^{even}$ are $\mathcal{L}$-analytic in $\OM$, thus in $\OM^+$.
By applying Theorem \ref{newAnalyticTh} to $\bv^{odd} \lvert_{\Lambda} \in l^{1,1}_{\INF}(\Lambda)\cap C^\mu(\Lambda;l_1)$ we recover the trace $\bv^{odd} \lvert_{L}$ on $L$.
Similarly, we recover the trace $\bv^{even} \lvert_{L}$ on $L$.

The Bukhgeim-Cauchy integral formula \eqref{BukhgeimCauchyFormula} extends $\bv^{odd}$ and $\bv^{even}$ from $\Lambda\cup L$ to $\OM^+$ as $\mathcal{L}$-analytic maps. 
From the uniqueness of an $\mathcal{L}$-analytic map with a given trace, we recovered 
\begin{align}\label{v_evenodd}
&\bv^{even}(z)= \B [ \bv^{even}\lvert_{\Lambda\cup L} ](z), \; \text{and} \; \bv^{odd}(z)=  \B [\bv^{odd}\lvert_{\Lambda\cup L} ](z),\; z\in \OM^+.
\end{align} Thus $\bv=\langle v_{0}, v_{-1}, v_{-2},v_{-3}, ... \rangle$ is recovered in $\OM^+$.

Now use the convolution formula \eqref{UintermsV} for $n=0$ and $-1$ to recover $u_0$ and $u_{-1}$ in $\OM^+$.

Finally, $f$ is reconstructed in $\OM^+$ by
\begin{align}\label{reconf_a>0}
f|_{\OM^+}(z)=2\re\{\del u_{-1}(z)\} +a(z)u_0(z),\quad z\in\OM^+.
\end{align}
\qed  %End the proof 

%%%%%%%%%%%%%%%%%%%%%%  Reconstruction Procedure %%%%%%%%%%%%%%%%%%%%

We summarize below in a stepwise fashion the reconstruction of $f$ in the convex hull $\OM^+$ of the boundary arc $\Lambda$. Recall that $L$ is the segment joining the endpoints of $\Lambda$.

\subsection*{Reconstruction procedure}
Consider the data $\bu \lvert_{\Lambda}$.
\begin{enumerate}
\item Using formula \eqref{VintermsU}, determine the traces $\bv^{odd} \lvert_{\Lambda}$ and $\bv^{even}	 \lvert_{\Lambda}$ on $\Lambda$.
 \item Recover the traces $\bv^{odd} \lvert_{L}$ and $\bv^{even} \lvert_{L}$ pointwise on $L$ as follows:
      \begin{enumerate}
               \item Using $\bv^{odd} \lvert_{\Lambda}$ and $\bv^{even}\lvert_{\Lambda}$, compute  by formula \eqref{Fanalytic},  the   
                          function $F_n$, for each $n\leq 0$. 
       \item Recover for each $n\leq 0$, the trace $v_{2n-1} \big \lvert_{L}$ by solving \eqref{Pminus_un}.
	\item Similarly, the trace $\bv^{even} \lvert_{L}$ is recovered pointwise on $L$.
      \end{enumerate}
 \item By Bukhgeim-Cauchy formula \eqref{v_evenodd}, extend $\bv^{odd}$ and $\bv^{even}$ from the boundary $\Lambda\cup L$ to $\OM^+$.
 \item Recover $u_0, u_{-1}$ in $\OM^+$ using \eqref{UintermsV} for $n=0, -1$.
 \item Recover $f|_{\OM^+}$ by formula \eqref{reconf_a>0}.
\end{enumerate}

%%%%%%%%%%%%%%%%%%%%%%%%%%%%%%%%%%%%%%%%%%%%%%%%%%%%%%%%%%%%%%%%
%%%%%%%%%%%%%%%%%%%% Numerical Section %%%%%%%%%%%%%%%%%%%%%%%%%

\section{Numerical reconstruction results}\label{numer_recon}

To illustrate the feasibility of the reconstruction procedure above, in this section we present two of its numerical implementations.  The rigorous analysis on the numerical methods presented below requires further study %from the stand point of theoretical numerical analysis 
and is left for a separate discussion.

Let $\Omega$ be the unit disk with following sub-domains depicted in Figure~\ref{fig:domain}:
\begin{align*}
R &= (-0.25,0.5)\times(-0.15,0.15),\\
B_1 &= \{(x-0.5)^2 + y^2 < 0.3^2\}, \\
B_2 &= \left\{(x+0.25)^2 + \left(y-\dfrac{\sqrt{3}}{4}\right)^2 < 0.2^2\right\},\\
B_3 &= \left\{x^2 + (y-0.6)^2 < 0.3^2\right\}.
\end{align*}
The measurement boundary is $\Lambda = \{ (\cos\theta,\sin\theta) \::\: 0 < \theta < \pi \} $, while the inaccessible  chord  $L = \{ -1 < x < 1, y=0 \}$.
\begin{figure}[h]
\begin{tikzpicture}[scale=2]
\draw [line width=2] (0,0) circle (1); % Omega
  \node at (0.9,0.9) {$\Omega$};

\draw [color=gray,fill=gray,line width=0] (-0.25,0.433013) circle (0.2); % omega2 (left)
\draw [color=gray,fill=gray,line width=0] (0,-0.6) circle (0.3); % omega3 (bottom)
  \node at (-0.4,-0.4) {$B_3$};

\draw [color=gray,fill=gray,line width=0] (-0.25,-0.15) rectangle (0.5,0.15); % R
  \node at (-0.4,-0.1) {$R$};

\draw [dotted,line width=2] (0.5,0) circle (0.3); % omega1 (right)
  \node at (0.5,0.45) {$B_1$};

\draw [dotted,line width=2] (-0.25,0.433013) circle (0.2); % omega2 (left)
    \node at (-0.25,0.8) {$B_2$};

\draw [->] (-1.2,0) -- (1.2,0);
\draw [->] (0,-1.2) -- (0,1.2);

\end{tikzpicture}
\caption{\label{fig:domain}Setting of numerical experiments.  The gray regions correspond to the support of the source $f$, whereas the dotted circles delineate regions of high attenuation. Note that $f$ is to be reconstructed only in the upper semi-disc.}
\end{figure}
%While the theoretical arguments in Theorem~\ref{reconstruct_f_OM+} requires some $C^{1,\mu}$-regular attenuation coefficient $a$ and source $f$, in the numerical setting they are discontinuous. 

We consider the equation \eqref{AttenRadonTEq_in_u} in the unit disk with the attenuation coefficient
\[
a(z) =
\begin{cases}
1, \quad &\text{in $B_1$};\\
2, \quad &\text{in $B_2$};\\
0.1, \quad &\text{otherwise},\\
\end{cases}
\]
and the source
\begin{align}\label{source}
f(z) =
\begin{cases}
2, \quad &\text{in $R$};\\
1, \quad &\text{in $B_2\cup B_3$};\\
0, \quad &\text{otherwise}.\\
\end{cases}
\end{align}
Note the contribution to the data coming from the source
supported in the lower half of the rectangle $R$ and in the ball $B_3$. In the inverse problem, the attenuation $a$ is known in the whole domain $\overline \OM$, while the source $f$ is unknown. We will recover $f$ of \eqref{source} in the upper semidisk $\Omega^+$. 

In the numerical experiments, the measurement data is obtained by numerical computation of the attenuated $X$-ray transform \eqref{Xaf} with \eqref{divbeam}
on $180$ and $360$ equi-spaced angular directions $\zeta\in\Lambda$ and $\theta \in \sph$, respectively.
The red curves in Figure~\ref{fig:measure} depict computed data
in the polar coordinate representation $\bigl\{ \bigl(u(\zeta,\theta),\theta\bigr) \::\: \theta \in \sph \bigr\}$
centered at $\zeta \in \Lambda$ which is indicated by a  black dot ($\bullet$).
The right graph is its magnification at $\zeta=(0,1)$. The pink disks in the left figure show the highly attenuating regions $B_1, B_2$.  

\begin{figure}[h]
\includegraphics[width=.4\textwidth,bb=135 125 900 735]{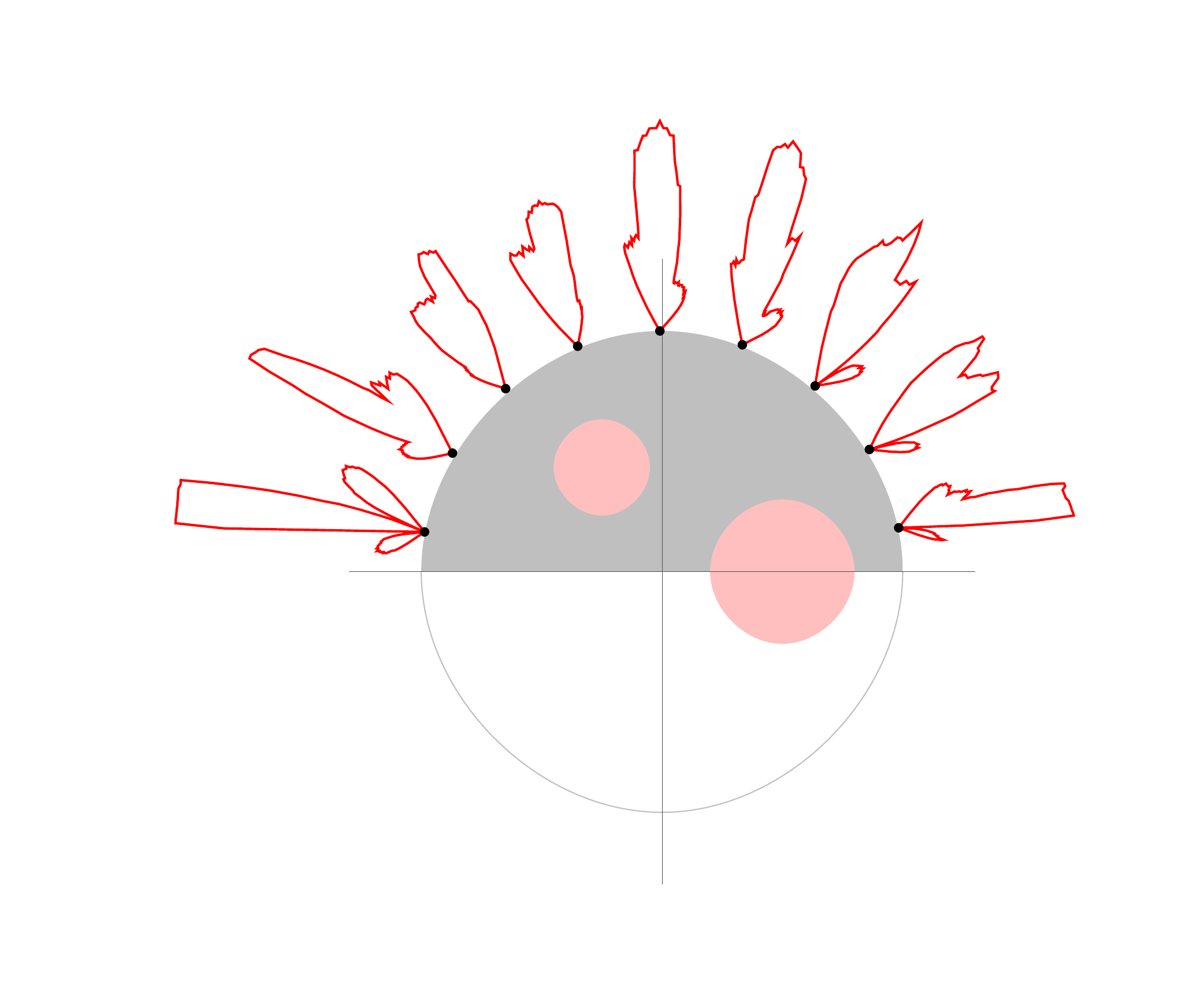}
\hspace{.8in}
\includegraphics[width=.11\textwidth,bb=135 0 215 215]{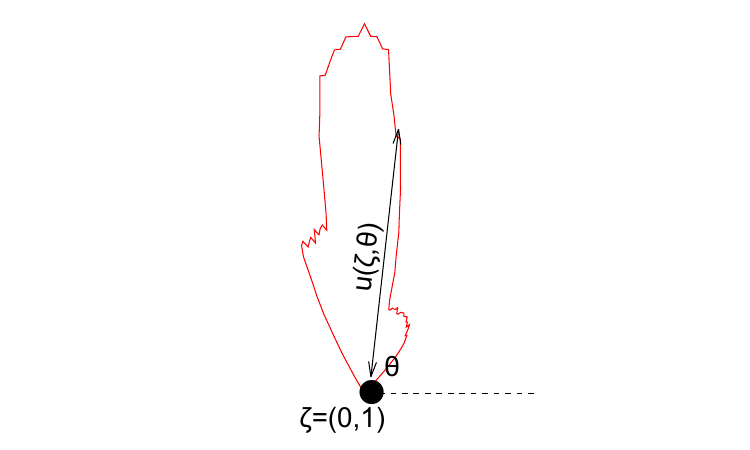}
\caption{\label{fig:measure}Partial measurement data $u(\zeta,\theta)|_{\Lambda\times\sph}$ obtained by numerical computation of
the attenuated Radon transform \eqref{u_Gam+}. The red curves are their polar coordinate representation $\{ (u(\zeta,\theta),\theta) \::\: \theta \in \sph\}$
and that at $\zeta=(0,1)$ is magnified in the right graph. The pink areas show the regions of (known) higher absorption.}
\end{figure}

The lowest negative Fourier mode used in  \eqref{AfEq} is $u_{-64}$ , and the series (e.g.\ in \eqref{BukhgeimCauchyFormula}) are truncated with finite terms. 

To calculate each of the occurring integrals, including in the integrating factor \eqref{ehEq}, we adopt the composite mid-point rule with $100$ sampling points. Particularly, a method in \cite{fktSIIMS20} is employed to calculate integrations in the sense of principal value. Moreover, the integrating factors \eqref{ehEq} are calculated with $360$ equi-spaced angular directions $\fii$. 

The reconstruction requires $\bu|_{\Lambda\cup L}$. While $\bu|_\Lambda$ is known from the data, we obtain each component $\left(u_n|_L\right)$ of $\bu|_L$ by numerically solving  \eqref{Pminus_un}. 
 %Even though $I-\i H_t$ is injective, it has an open and dense range in $L^2(-l,l)$ \cite{koppelmanPincus58,okadaElliott91}, yielding the inversion unstable. Discretization, however, stabilizes the inversion.
The numerical solution $u_n|_L$  is found for each mode $n$ via the discretized collocation
$\ds
[I-\i H_t]u_n(x_i)
\approx u_n(x_i) - \dfrac{\i}{\pi}\sum_{j\neq i}\dfrac{u_n(x_j)}{x_i - x_j}\Delta x,
$
which avoids the singularities. The only regularization used in solving  \eqref{Pminus_un} is implicit by the discretization. For illustration, computed solutions of $u_{n}|_L$, for $-5\leq n\leq0$   are shown in Figure~\ref{fig:vn}, where left and right graphs are their real and imaginary parts respectively.
\begin{figure}[h]
\begin{minipage}{.48\textwidth}
\includegraphics[width=\textwidth]{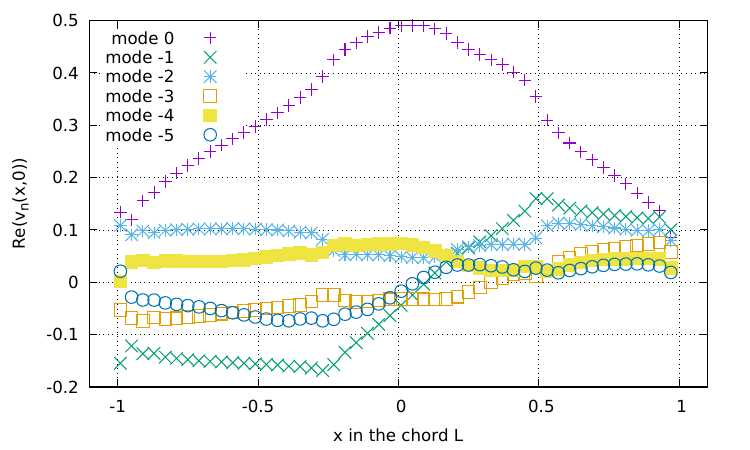}
\end{minipage}
\begin{minipage}{.48\textwidth}
\includegraphics[width=\textwidth]{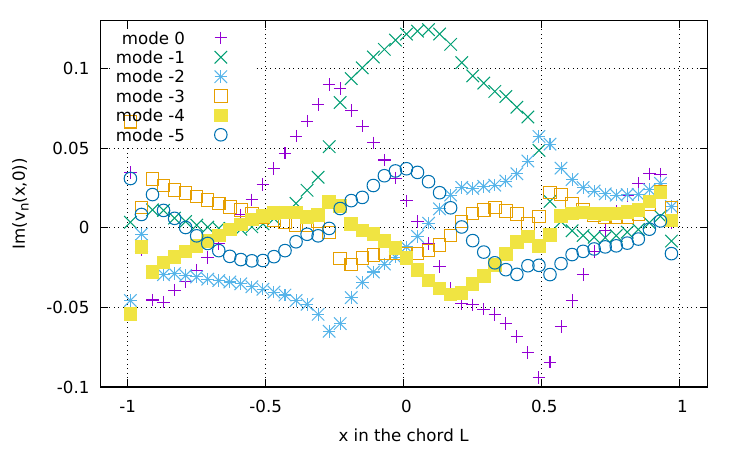}
\end{minipage}
\caption{\label{fig:vn}Numerical solutions on $L$ to the singular integral equation \eqref{Pminus_un}, from $u_{0}$ to $u_{-5}$, real parts (left) and imaginary parts (right).}
% Note: Kamran uses only odd indices, while Hiroshi uses full indices
\end{figure}

When recovering $f$ via \eqref{recon}  we approximate derivatives via the standard central finite difference on a regular grid with $\Delta x = \Delta y = 0.01$. Note that in our numerical experiments we avoid the inverse crime, since the corresponding forward problem is calculated by the line integral on the characteristics.

Under the setting above, the elapsed times for the forward and inverse problems are
approximately $5$ and $180$ seconds respectively, while using the OpenMP parallel computation
on two Xeon E5-2650 v4 (2.20GHz, 12 cores) processors (totally 24 cores).

%%-----------------------------------------
\begin{figure}[h]
\begin{minipage}[m]{.7\textwidth}
\centering
\includegraphics[width=\textwidth,bb=75 45 345 165]{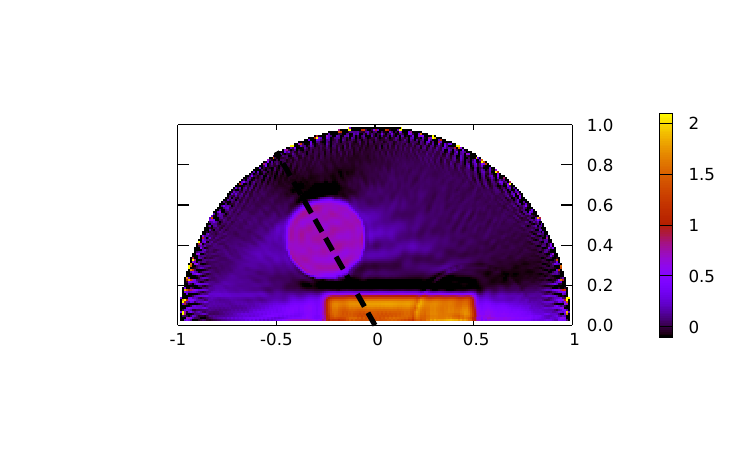}
\subcaption{Profiles of reconstructed source $f$ in $\Omega^{+}$}
\end{minipage}
\hfill
\begin{minipage}{.29\textwidth}
\includegraphics[width=\textwidth,bb=60 0 185 150]{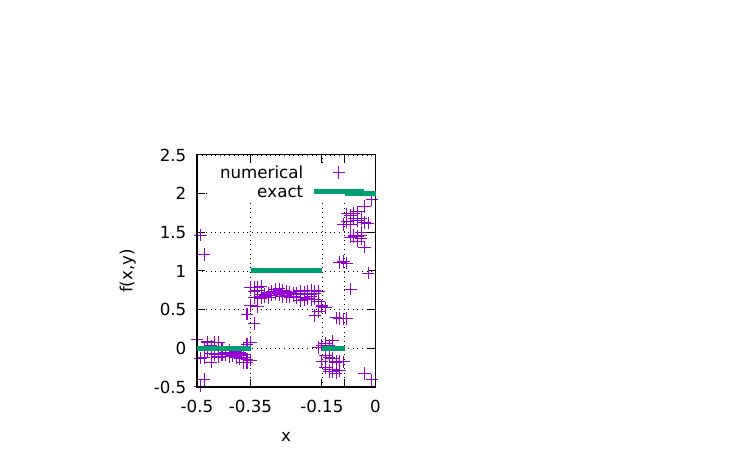}
\subcaption{Section on $y=-\sqrt{3}x$.}
\end{minipage}
\caption{\label{fig:reconst}Numerically reconstructed source $f$ in $\Omega^{+}$ (left), and
its section on $y=-\sqrt{3}x$ (right), which is indicated by the dotted line in the left figure.}
\end{figure}

Figure~\ref{fig:reconst} on the left shows the numerically reconstruction source $f$ in $\OM^+$.  Note that the singular support of $f$  (the discontinuities across $\partial R$ and $\partial B_2$) is clearly detected. This is to be expected, since, at least for an infinitely smooth attenuation, the conormal singularities are known to be microlocally stable. However, in here we do a quantitative reconstruction. To assess its accuracy, Figure \ref{fig:reconst} on the right shows the section of reconstructed $f$ on the line $y=-\sqrt{3}x$ (the dotted segment in the left figure). We note a reasonably quantitative agreement with the exact solution. 
The discrepancy is due in part to the instability of the singular integral equation \eqref{Pminus_un}, %inherent to the inversion of a resolvent operator through a point on the spectrum,  
but also to the simplicity of  the regularization used to stabilize it (merely by discretization). 
%and it will be discussed in a subsequent paper.

%%%%%%%%%%%%%%%%%%%%%%%%%%%%%%%%%%%%%%%%%%%%%%%%%%%%%%%%%%%%%%%%%%%%%%
% (Message by Hiroshi)
% If reconstruction from noisy data is not needed, please erase after this line
%%%%%%%%%%%%%%%%%%%%%%%%%%%%%%%%%%%%%%%%%%%%%%%%%%%%%%%%%%%%%%%%%%%%%%
To check the robustness of the proposed method, we performed a numerical reconstruction from noisy boundary data, where in addition to the computational error caused by numerical quadrature to \eqref{Xaf} with \eqref{divbeam}, 
we introduce some artificial error. The noisy boundary data shown in Figure~\ref{fig:measure:noise10} is generated with a pseudo-random 
number routine in the programming language C++ and has $10.9\%$ error  (in the relative $L^2$ norm) added to the noiseless data of Figure~\ref{fig:measure}.
\begin{figure}[t]
\includegraphics[width=.4\textwidth,bb=135 125 900 735]{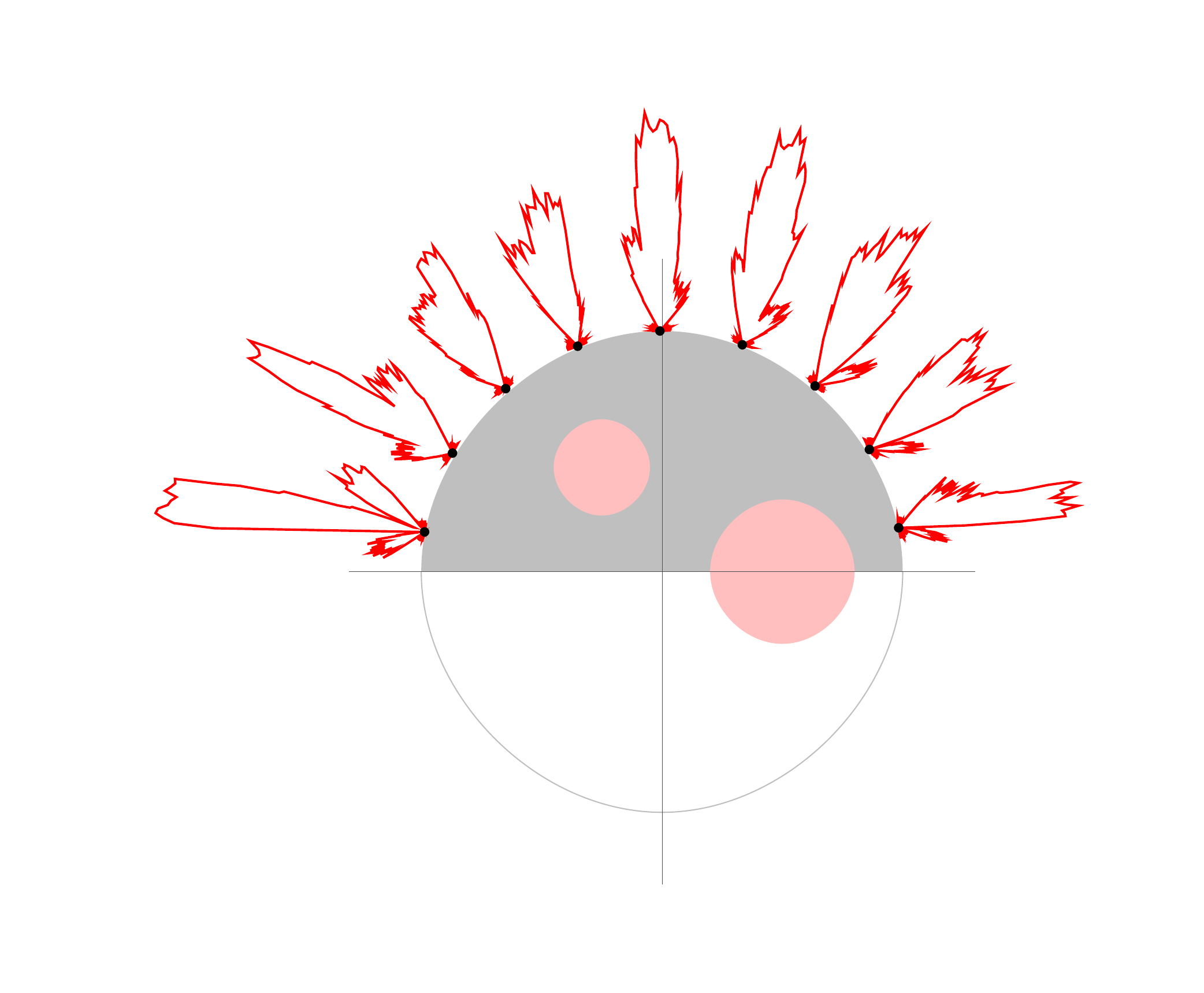}
\hspace{.8in}
\includegraphics[width=.11\textwidth,bb=135 0 215 215]{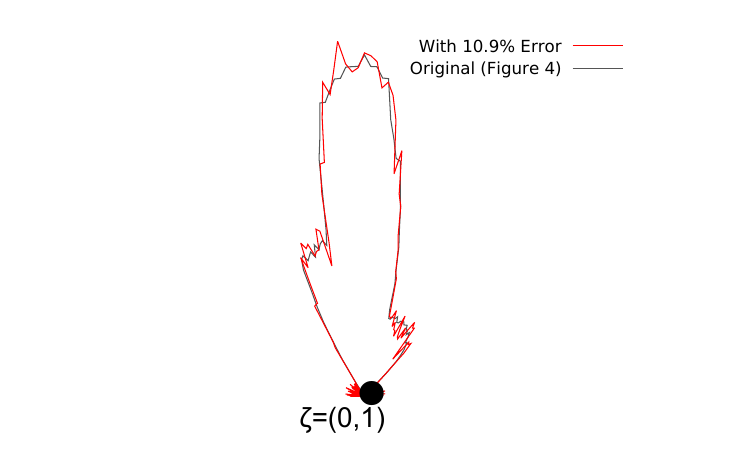}
\caption{\label{fig:measure:noise10}Partial measurement data $u(\zeta,\theta)|_{\Lambda\times\sph}$ with
$10.9\%$ noise in the relative $L^2$ norm, depicted in the same manner as in Figure~\ref{fig:measure}.}
\end{figure}
The reconstructed source $f$ from noisy data is shown in Figure~\ref{fig:reconst:noise10}. 
\begin{figure}[b]
\begin{minipage}[m]{.7\textwidth}
\centering
\includegraphics[width=\textwidth,bb=75 45 345 165]{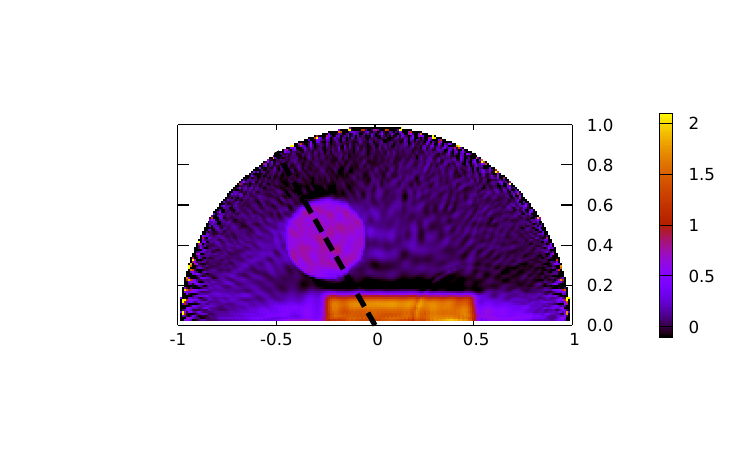}
\subcaption{Profiles of reconstructed source $f$ in $\Omega^{+}$}
\end{minipage}
\hfill
\begin{minipage}{.29\textwidth}
\includegraphics[width=\textwidth,bb=60 0 185 150]{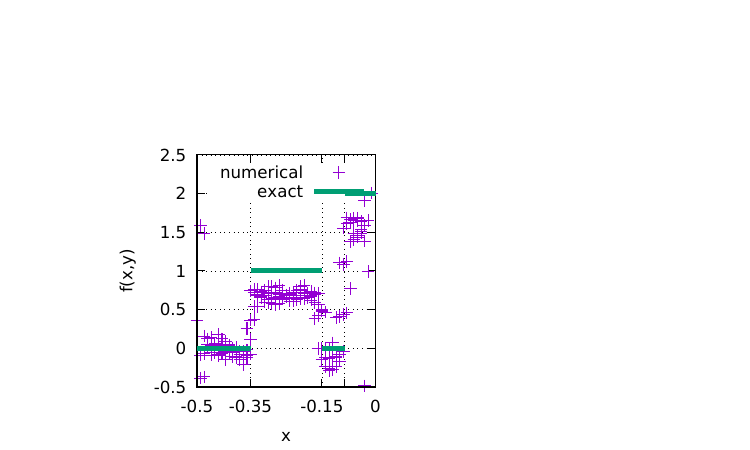}
\subcaption{Section on $y=-\sqrt{3}x$}
\end{minipage}
\caption{\label{fig:reconst:noise10}Numerically reconstructed source $f$ from
partial measurement data with $10.9\%$ noise shown in Figure~\ref{fig:measure:noise10}.}
\end{figure}
Its similarity with the reconstructed source from the noiseless data in Figure~\ref{fig:reconst} indicates the robustness of the proposed procedure. Moreover, since the numerical experiment considered discontinuous coefficients,  
the numerical results also suggest the possibility of reconstruction under less regular assumptions than the ones required by Theorem~\ref{reconstruct_f_OM+}.

%%%%%%%%%%%%%%%%%%%%%%%%%%%%%%%%%%%%%%%%%%%%%%%%%%%%%%%%%%%%%%%%%%%%%%
% End of reconstruction from noisy data
%%%%%%%%%%%%%%%%%%%%%%%%%%%%%%%%%%%%%%%%%%%%%%%%%%%%%%%%%%%%%%%%%%%%%%

%%%%%%%%%%%%%%%%%%%%%%%%%%%%%%%%%%%%%%%%%%%%%%%%%%%%%%%%%%%%%%%%
\section*{Acknowledgment}
We thank  P.~Stefanov and A.~Katsevich for useful discussions. We are also grateful to the anonymous referee for the constructive criticism on a previous version of this work. 
The work of H.~ Fujiwara was supported by JSPS KAKENHI Grant Numbers JP16H02155 and JP20H01821. 
The work of K.~ Sadiq  was supported by the Austrian Science Fund (FWF), Project P31053-N32. 
The work of A.~Tamasan  was supported in part by the NSF grant DMS-1907097.
%%%%%%%%%%%%%%%%%%%%%%%%%%%%%%%%%%%%%%%%%%%%%%%%%%%%%%%%%%%%%%%%

%%%%%%%%%%%%%%%%%%%%%%%%%%%%%%%%%%%%%%%%%%%%%%%%%%%%%%%%%%%%%%%%

%%%%%%%%%%%%%%%%%%%%%%%%%%%%%%%%%%%%%%%%%%%%%%%%%%%%%%%%%%%%%%%%%%%%%%%%%%
\end{document}